\documentclass{amsart} 
\usepackage[utf8]{inputenc}
\usepackage{amssymb}
\usepackage{amsmath}
\usepackage{amsthm}
\usepackage{mathtools}
\usepackage{multirow}
\usepackage{thmtools}

\newtheorem{conjecture}{Conjecture}
\newtheorem{theorem}{Theorem}
\newtheorem{remark}{Remark}

\usepackage{tikz}
\usetikzlibrary{calc}
\usetikzlibrary{braids}
\usepackage[colorlinks=true, citecolor=blue, backref=page]{hyperref}

\begin{document}

\title[The Vol-Det conjecture for highly twisted links]{The Vol-Det conjecture for highly twisted alternating links}

\author[A.~Egorov]{Andrei~Egorov}
\address{Sobolev Institute of Mathematics, Siberian Branch of Russian Academy of Sciences, Novosibirsk, 630090, Russia  
\& 
Department of  Mechanics and Mathematics, 
Novosibirsk State University, Novosibirsk, 630090, Russia  
\& Regional Scientific and Educational Mathematical Center, Tomsk State University, Tomsk, 634050, Russia}
\email{\href{mailto:a.egorov2@g.nsu.ru}{a.egorov2@g.nsu.ru}}

\author[A.~Vesnin]{Andrei~Vesnin} 
\address{Sobolev Institute of Mathematics, Siberian Branch of Russian Academy of Sciences, Novosibirsk, 630090, Russia  
\&
Department of  Mechanics and Mathematics, 
Novosibirsk State University, Novosibirsk, 630090, Russia  
\& Regional Scientific and Educational Mathematical Center, Tomsk State University, Tomsk, 634050, Russia}
\email{\href{mailto:vesnin@math.nsc.ru}{vesnin@math.nsc.ru}}

\begin{abstract}
The Vol-Det Conjecture, formulated by Champanerkar, Kofman and Purcell, states that there exists a specific inequality connecting the hyperbolic volume of an alternating link and its determinant. Among the classes of links for which this conjecture holds are all alternating hyperbolic knots with at most 16 crossings, 2-bridge links, and links that are closures of 3-strand braids. In the present paper, Burton's bound on the number of crossings for which the Vol-Det Conjecture holds is improved for links with more than eight twists. In addition, Stoimenow's  inequalities between hyperbolic volumes and determinants are improved for alternating and alternating arborescent links with more than eight twists.
\end{abstract}

\thanks{This work was supported by the Ministry of Science and Higher Education of Russia (agreement no. 075-02-2024-1437).} 


\renewcommand{\baselinestretch}{1.2}

\date{\today}

\maketitle

\section{Introduction}

In recent decades, invariants constructed by using methods from both algebraic topology and hyperbolic geometry have been  actively studied and used in knot theory. The first family of these invariants includes, for example,  the Alexander polynomial, the Jones polynomial, the knot determinant and the knot signature (see~\cite{Lic97}). The second family includes geometric invariants associated with the complement of a knot when it is a hyperbolic manifold, such as the volume of the complement, the trace field of the fundamental group, and the group of isometries (see~\cite{BePe92, Th80}). Comparing the values of these different invariants leads to conjectures, including those made using artificial intelligence (see~\cite{AI2021}) about the relationships between invariants of different kinds. Apparently, the first example of such a conjecture was  Dunfield's hypothesis concerning the relation between the value of the Jones polynomial $V_K(t)$ for a given knot $K$ at $t=-1$ and the hyperbolic volume of its complement $S^3 \setminus K$ (see ~\cite{Dun00}). In the recent paper~\cite{DJLT24}, Davies, Juh\'{a}sz, Lackenby and Toma\v{s}ev established a connection between the knot signature and the cusp geometry for the case of hyperbolic knots. 
 
 In the present paper we discuss a conjecture that relates the volume  $\operatorname{vol} (K)$ of a hyperbolic 3-manifold, which is the complement of the link $K$, to its determinant $\det(K)$. This conjecture is known in low-dimensional topology as the Vol-Det Conjecture~\cite{CKP16, CKL19}. 

The determinant $\det(K)$ of the link $K$ is one of the best known invariants that can be obtained from the link  diagram. Namely, $\det(K)$ can be computed in one of the following ways:
$$
\det (K)  = | H_1 (\Sigma_2 (K), \mathbb Z) |  =  | \det (G(K)) |  =  | \det (M + M^{T}) | = | \Delta_K (-1)| = | V_K (-1) |,  
$$
where $\Sigma_2(K)$ is a double cover of the sphere $S^3$, branched over $K$, and $|H_1 (\Sigma_2(K),\mathbb Z) |$ is the order of its first homology group; $G(K)$ is the Goeritz matrix for $K$; $M$ is the Seifert matrix for $K$ and $M^T$ is the transposed matrix; $\Delta_K(t)$ is the Alexander polynomial of the link $K$; $V_k(t)$ is the Jones polynomial of the link $K$, see more in ~\cite{Lic97}. 
 
A link $K$ in $S^3$ is called \emph{hyperbolic} if its complement $S^3\setminus K$ is a three-dimensional hyperbolic manifold, see ~\cite{Th80}. The volume of the hyperbolic link $K$ is understood as the volume of its complement, $\operatorname{vol}(K) = \operatorname{vol} (S^3 \setminus K)$. Recall that a link $K$ in the three-dimensional sphere $S^3$ is \emph{alternating} if it has an alternating diagram, i.e. such a diagram that, as one moves along each component, overcrossings and undercrossings alternate. According to~\cite{Bur20}, among $352\,152\,252$ prime knots admitting diagrams with at most $19$ crossings there are $352\,151\,858$ hyperbolic knots and $51\,280\,967$ of them are alternating. 

The Vol-Det Conjecture was formulated by Champanerkar, Kofman and Purcell and it states the following.

\begin{conjecture} \cite{CKP16} \label{conj1}
Let $K$ be an alternating hyperbolic link. Then 
\begin{equation}
\operatorname{vol} (K) < 2 \pi \cdot \log \det (K).
\end{equation}
\end{conjecture}

In~\cite{CKP16}, it was shown that the conjecture~\ref{conj1} is true for all hyperbolic alternating knots admitting diagrams with no more than $16$ crossings, as well as for several infinite families of knots. In addition, it was demonstrated that the constant $2\pi$  in the conjecture~\ref{conj1} is accurate, that is for any $\alpha <  2\pi$ there exists an alternating hyperbolic knot $K_{\alpha}$ such that $\operatorname{vol} (K_{\alpha}) > \alpha  \cdot \log \det (K_{\alpha})$. In~\cite{CKL19}, based on the calculation of the Mahler measure for some polynomials in two variables, it was shown that the conjecture~\ref{conj1} is true for several infinite families of alternating links.  

Burton proved that the conjecture~\ref{conj1} is true for hyperbolic $2$-bridge links, see~\cite[Th.~3.5]{Bur18}, and for hyperbolic alternating links which are closures of $3$-strand braids, see~\cite[Th.~4.1]{Bur18}.

\smallskip

Next, we will consider estimates of volumes and determinants using the concept of twist in the diagram. 

Let $K$ be a link in $S^3$, and $D$ be a diagram of $K$. The \emph{twist} in the diagram $D$ is a connected set of bigons in $D$ arranged in a line, which is maximal in the sense that it is not a part of a longer connected set of bigons, or it is a single crossing incident to regions of the diagram that are not bigons. The \emph{number of twists} in the diagram $D$ is denoted by $t(D)$. For example, consider a knot which is the closure of a $4$-strand braid $\beta = \sigma_1^2 \sigma_2^{-3} \sigma_3^2 \sigma_1 \sigma_2^{-2} \sigma_3$ shown in Figure~\ref{fig1}. This knot is notated by $11_{120}$ in~\cite{Snap} and by $11a47$ in~\cite{LiMo}. Its diagram has six twists which are marked by red dotted closed curves in the figure. Since the volume and the determinant of the knot $10_{120}$ are tabulated in~\cite{Snap} and~\cite{LiMo}, namely, $\operatorname{vol}(10_{120}) = 15.597714$ and $\det(10_{120}) = 117$, it is easy to make sure that the conjecture~\ref{conj1} is true for it.

\begin{figure}[h]
	\begin{center}
		\scalebox{0.8}{
			\begin{tikzpicture}
			\pic[
			rotate=90,
			braid/.cd,
			every strand/.style={ultra thick},
			strand 1/.style={black},
			strand 2/.style={black},
			strand 3/.style={black},
			strand 4/.style={black},]
			at (0,0) {braid={s_3^{-1} s_3^{-1} s_2^1 s_2^1 s_2^1 s_1^{-1} s_1^{-1}  s_3^{-1} s_2^1 s_2^1 s_1^{-1} }};
			\draw [ultra thick, black] (0,4) -- (11.5,4);
			\draw [ultra thick, black] (0,3.5) -- (11.5,3.5);
			\draw [ultra thick, black] (0,-.5) -- (11.5,-.5);
			\draw [ultra thick, black] (0,-1) -- (11.5,-1);	
			\draw[ultra thick, black] (0,4) arc (90:270:1);
			\draw[ultra thick, black] (11.5,4) arc (90:-90:1);	
			\draw[ultra thick, black] (0,3.5) arc (90:270:0.25);
			\draw[ultra thick, black] (11.5,3.5) arc (90:-90:0.25);			
			\draw[ultra thick, black] (0,0) arc (90:270:0.25);
			\draw[ultra thick, black] (11.5,0) arc (90:-90:0.25);
			\draw[ultra thick, black] (0,1) arc (90:270:1);
			\draw[ultra thick, black] (11.5,1) arc (90:-90:1);
			\draw[ultra thick, dotted, red] (1.25,2.5) ellipse (1cm and 0.65cm);
			\draw[ultra thick, dotted, red] (7.75,2.5) ellipse (0.7cm and 0.65cm);
			\draw[ultra thick, dotted, red] (3.75,1.5) ellipse (1.4cm and 0.65cm);
			\draw[ultra thick, dotted, red] (9.25,1.5) ellipse (1cm and 0.65cm);
			\draw[ultra thick, dotted, red] (6.25,0.5) ellipse (1cm and 0.65cm);
			\draw[ultra thick, dotted, red] (10.75,0.5) ellipse (0.7cm and 0.65cm);
			\end{tikzpicture}
		}
	\end{center}
	\caption{An alternating diagram of the knot $11_{120}$ with six twists.} \label{fig1}
\end{figure}
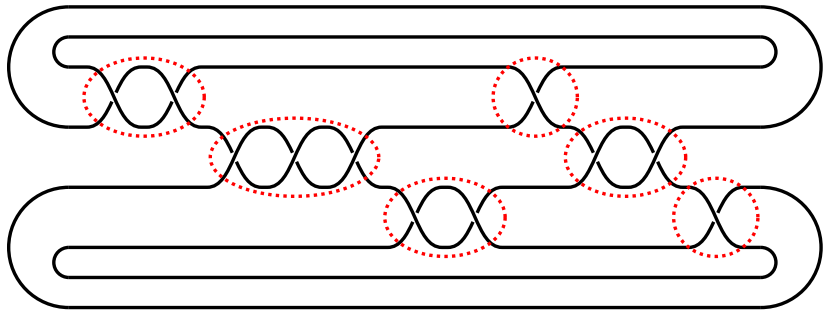

Throughout the paper we will suppose that the diagrams we consider are prime and twist-reduced. The diagram is \emph{prime} if any simple closed curve which meets two edges of the diagram transversely must bound a region of the diagram with no crossings. The diagram is \emph{twist-reduced} if any simple closed curve intersect the diagram transversely in four edges, with two points of intersection adjacent to one crossing and the other two adjacent to another crossing, then that curve must bound a region of the diagram consisting of a (possibly empty) collection of bigons arranged end to end between the crossings, see~\cite{Lac04, Pur07, Sto07} for details. 
If the diagram of a hyperbolic link is not prime, then some crossings are redundant and can be removed to obtain a prime diagram.
Lackenby proved~\cite{Lac04} that if the connected prime alternating diagram is not twist-reduced, then a series of flypes will decrease the number of twist regions of the diagram, until we are left with a twist-reduced diagram.

In this paper we consider a class of alternating links admitting reduced alternating diagrams with more than eight twists. For this class of links the following results are obtained:  
\begin{itemize}
\item[(1)] an improvement of the Burton's bound~\cite[Th.~5.1]{Bur18} on the number of crossings of links for which the Vol-Det Conjecture holds (see theorem~\ref{thm1}); 
\item[(2)] an improvement of the Stoimenow's inequality~\cite[Th.~1.1]{Sto07} between hyperbolic volume and determinant for alternating links (see theorem~\ref{thm2});
\item[(3)] an improvement of the Stoimenow's inequality~\cite[Prop.~6.3]{Sto07} between hyperbolic volume and determinant for links admitting reduced alternating arborescent diagrams (see theorem~\ref{thm3}).  
\end{itemize}
The results are based on using the bound for the volume of  the link complement in terms of the number of twists from~\cite{VeEg24} instead of~\cite{Lac04}.  Moreover, in theorem~\ref{thm1} we use the bound for the link determinant obtained by Ito in~\cite{Ito23}.

\section{The Vol-Det Conjecture for highly twisted links}

We begin by recalling three results which show that the determinant of a link can be bounded from below by the  function depending only on the number of twists in its diagram, as well as by the functions depending only on number of twists and the number of crossings. Here and further we will denote by $\gamma$ a number such that $1/\gamma$ is the real positive root of the equation $x^3 (x+1)^2 = 1$. It is easy to see that such a root is unique and equal to $0.701607$ up to six decimal places. So we have
\begin{equation}
\gamma = 1.425299. \label{gamma}
\end{equation}

In~\cite[Th.~4.3]{Sto07} Stoimenow showed that if $t(D)$ is the number of twists in the given alternating twist-reduced diagram $D$ of a link $K$, then there is an inequality  
\begin{equation}
	\det (K) \geq 2 \gamma^{t(D)-1}.  \label{eqn:Sto}
\end{equation} 

In~\cite[Th.~5.1]{Bur18} Burton improved inequality~(\ref{eqn:Ito}) by considering not only the number of twists, but also the number of crossings. Namely, let $K$ be an alternating hyperbolic link with the reduced alternating twist-reduced diagram $D$ having $t(D)$ twists and $c(D)$ crossings.  
Then
\begin{equation}
	\det (K) \geq 2 \gamma^{t(D)-1} + c(D) - t(D). \label{eqn:Burton}
\end{equation}   

In~\cite[Th.~4.3]{Ito23} Ito improved inequality~(\ref{eqn:Sto}) in the following way. Let $D$ be a reduced alternating twist-reduced diagram of a link $K$ with $t(D) \geq 2$. Then 
\begin{equation}
\operatorname{det} (K) > 2 \gamma^{-1} \left( \gamma^{t(D)} + (c(D) - t(D)) \, \gamma^{(t(D) - 1)/2} \right). \label{eqn:Ito}
\end{equation}  

To compute the volumes of polyhedra in the three-dimensional hyperbolic space $\mathbb H^3$ and the volumes of three-dimensional hyperbolic manifolds, the Lobachevsky function $\Lambda(\theta)$ is usually used. This function was introduced by Milnor in~\cite{Mil82} and is given by the formula   
\begin{equation*}
	\Lambda (\theta) = - \int_0^{\theta} \ln | 2 \sin (t) | dt.
\end{equation*}

In the following we will use the value $v_{\text{tet}}  = 3\Lambda (\pi / 3)$, which is equal to the volume of a regular (all dihedral angles are equal) hyperbolic ideal (all vertices are in $\partial\mathbb H^3$) tetrahedron. Such a tetrahedron has all dihedral angles equal to $\pi/3$ and it has the largest volume among all hyperbolic tetrahedra with finite volume. Up to six digits after the decimal point, we have
\begin{equation}
	v_{\text{tet}} = 1.014941. \label{volume}
\end{equation}

In~\cite{Lac04} Lackenby, Agol and Thurston established an inequality that allows us to bound the volume of a link from above in terms of the number of twists. Namely, suppose the diagram $D$ of a link $K$ has $t(D)$ twists, then  
\begin{equation}
	\operatorname{vol} (K) \leq 10 \, v_{\text{tet}} \, (t(D) - 1). \label{eqn:Lackenby}
\end{equation}
They also showed that there is a sequence of links $K_i$ such that $\operatorname{vol}(K_i) /t(D_i)\to 10 v_{\text{tet}}$. The bound~(\ref{eqn:Lackenby}) is related to the decomposition of the complements of the augmented links into ideal right-angled hyperbolic polyhedra. For more information on this class of polyhedra see~\cite{Ve17, VeEg2020}. 

Using the inequalities~(\ref{eqn:Burton}) and~(\ref{eqn:Lackenby}) Burton obtained the following result~\cite[Th.~5.1]{Bur18}.  Let $K$ be an alternating hyperbolic link, and its reduced alternating diagram $D$ has $t(D)$ twists and $c(D)$ crossings. Let us denote 
\begin{equation}
\xi = \exp \left( \frac{5 v_{\text{tet}}}{\pi} \right). \label{eqn:xi}
\end{equation}   
If the inequality 
\begin{equation}
	c(D) \geq t(D) + \xi^{t(D) - 1} - 2 \gamma^{t(D) - 1}  \label{eqn6}
\end{equation}
is satisfied, then the conjecture~\ref{conj1} is true for $K$. 

The following theorem shows that for links with diagrams containing more than eight twists, the bound~(\ref{eqn6}) can be improved by using a stronger estimates from~\cite{VeEg24} and~\cite{Ito23} instead of~(\ref{eqn:Lackenby}) and~(\ref{eqn:Burton}).

\begin{theorem} \label{thm1}
	Let $K$ be an alternating hyperbolic link, and let its reduced alterna\-ting twist-reduced diagram $D$ have $t(D) > 8$ twists and $c(D)$ crossings. If    
\begin{equation}
c(D) \geq t(D) + \frac{\xi^{t(D) - 1.4}}{2 \gamma^{\frac{t(D)-3}{2}}} - \gamma^{\frac{t(D) + 1}{2}}, \label{eqn:main}
\end{equation}
then the Vol-Det Conjecture  is true for $K$. 
\end{theorem}

\begin{proof} By virtue of~\cite[Th.~4.3]{Ito23} there is an inequality~(\ref{eqn:Ito}):  
$$
\det (K) \geq 2 \gamma^{t(D) - 1} + 2 (c(D) - t(D)) \, \gamma^{\frac{t(D) - 3}{2}}.
$$
Using the condition (\ref{eqn:main}) we get
$$
\det (K) \geq \xi^{t(D) - 1.4}, 
$$
where $\xi$ is given by (\ref{eqn:xi}). So,   
	$$
	\log \det (K) \geq (t(D) - 1.4) \frac{5 v_{\rm tet}}{\pi},
	$$
	from where
	\begin{equation}
		2 \pi \log \det (K) \geq 10 \, v_{\rm tet} \, (t(D) -1.4). \label{eqn:100}
	\end{equation} 
	As it was shown in~\cite{VeEg24}, if $t(D) > 8$, then there is an inequality 
	\begin{equation}
		\operatorname{vol} ( K) \leq 10 \, v_{\rm tet} \, (t(D) - 1.4). \label{eqn:VE}
	\end{equation}
	Comparing the inequalities (\ref{eqn:100}) and (\ref{eqn:VE}), we see that the conjecture~\ref{conj1} is true for a link $K$. 
	The theorem is proved. $\Box$ \end{proof}

\begin{remark} {\rm 
Let us illustrate the difference between the estimates (\ref{eqn6}) and (\ref{eqn:main}) for the case $t(D) = 9$. Using the approximations  $\gamma = 1.425299$ and $\xi = 5.029546$, for the first estimate we obtain that the conjecture~\ref{conj1} is true when $c(D) \geq 409\,453$, and for the second estimate when $c(D) \geq 37\,055$.
	}
\end{remark} 

\section{The volume bound via the determinant of link}

Using inequalities (\ref{eqn:Sto}) and (\ref{eqn:Lackenby}) Stoimenow~\cite[Th.~1.1]{Sto07} showed that if $K$ is a non-trivial non-split alternating link, then 
\begin{equation}
	\operatorname{vol}(K) \leq A\cdot \log \det(K) - B, \qquad\text{where} \quad A = \frac{10 \, v_{\text{tet}}}{\log \gamma}, \quad B = \frac{10\, v_{\text{tet}} \log 2}{\log \gamma}. \label{eqn2}
\end{equation} 

The following theorem shows that for links whose diagrams have more than eight twists the bound (\ref{eqn2}) can be improved by using the inequality~(\ref{eqn:VE}) instead of the  inequality (\ref{eqn:Lackenby}).

\begin{theorem} \label{thm2} 
	Let $K$ be a non-trivial non-split alternating link, and its reduced alternating twist-reduced diagram $D$ have $t(D) > 8$ twists. Then the following inequality holds  
\begin{equation}
\operatorname{vol}(K) \leq A\cdot\log \det(K) - C, \label{eqn4}
\end{equation}		
with constants  
\begin{equation} 
A = \frac{10\, v_{\text{\rm tet}}}{\log \gamma} \quad \text{and} \quad C = \frac{10\, v_{\text{\rm tet}} \log 2}{\log \gamma} + 4 v_{\text{\rm tet}}. 
\end{equation} 
\end{theorem}

\begin{proof} 
	By virtue of~\cite[Th.~4.3]{Sto07} there is an inequality (\ref{eqn:Sto}):
	$$
	\det (K) \geq 2 \gamma^{t(D) - 1}. 
	$$
	Therefore,
	$$
	(t(D) - 1) \log\gamma\leq \log\det(K) - \log 2.
	$$
	Using the formula (\ref{eqn:VE}) we get
	$$
	\operatorname{vol} (K) \leq 10 \, v_{\rm tet} \, (t(D) - 1.4) \leq \frac{10 \, v_{\rm tet}}{\log \gamma} \log \det (K) - \frac{10 \, v_{\rm tet} \log 2}{\log \gamma} - 4 \, v_{\rm tet}.
	$$
	The theorem has been proved. $\Box$ \end{proof}

\begin{remark}
By replacing $v_{\text{tet}}$ and $\gamma$ with their approximate values in the inequality (\ref{eqn2}), we get  
\begin{equation}
\operatorname{vol} (K) \leq  28.639760  \cdot \log \det (K) - 15.791802,  	
\end{equation} 
while from the inequality (\ref{eqn4}) we obtain 
\begin{equation}
\operatorname{vol} (K) \leq 28.639760  \cdot \log \det (K) - 19.851568.  
\end{equation} 
\end{remark}	
	
\section{The case of arborescent links}

The proof of the inequality (\ref{eqn:Sto}) in~\cite{Sto07} is based on the correspondence between alternating diagrams and the chessboard coloring graphs. It was observed in~\cite{Sto07} that (\ref{eqn:Sto}) can be improved for the class of alternating arborescent links. For a plane graphs consider the \emph{doubling} of an edge (installing an edge connecting its terminal vertices) and the \emph{bisection} of an edge (putting a vertex of valency two in it). These two operations are presented in Figure~\ref{fig2}, where edges are presented by dotted lines (which are red in the colored version). Other edges may be incident from the left- and rightmost vertices on both sides, and the rest of the graphs on both sides are assumed to be equal.  
For a graph $\mathcal G$ let $\langle \mathcal G \rangle$ be the family of graphs obtained from $\mathcal G$ by repeated edge bisections and doublings. 
\begin{figure}[h]
	\begin{center}
		\scalebox{1.0}{
			\begin{tikzpicture}
			\draw[black, very thick,<-] (-1.5,0.5) -- (-0.5,0.5); 
			\node[above] at (-1,0.5) {\text{bisection}};
			\node[below] at (-1,0.5) {\text{of edge}};
			\pic[
			rotate=90,
			braid/.cd,
			every strand/.style={ultra thick},
			strand 1/.style={black},
			strand 2/.style={black},]
			at (-4.5,0) {braid={s_1^{-1} s_1^{-1}}};
			\filldraw[red] (-4.5,0.5) circle (0.08); 
			\filldraw[red] (-3.25,0.5) circle (0.08); 
			\filldraw[red] (-2,0.5) circle (0.08); 
			\draw [ultra thick, dotted, red] (-4.5,0.5) -- (-3.9,0.5);
			\draw [ultra thick, dotted, red] (-3.6,0.5) -- (-2.9,0.5);
			\draw [ultra thick, dotted, red] (-2.6,0.5) -- (-2,0.5);
			\pic[
			rotate=90,
			braid/.cd,
			every strand/.style={ultra thick},
			strand 1/.style={black},
			strand 2/.style={black},]
			at (0,0) {braid={s_1^{-1} }};
			\draw [ultra thick, dotted, red] (0,0.5) -- (0.6,0.5);
			\draw [ultra thick, dotted, red] (0.9,0.5) -- (1.5,0.5); 
			\filldraw[red] (0,0.5) circle (0.08); 
			\filldraw[red] (1.5,0.5) circle (0.08); 
			\draw[black, very thick,->] (2,0.5)--(3,0.5);
			\node[above] at (2.5,0.5) {\text{doubling}};
			\node[below] at (2.5,0.5) {\text{of edge}};
			\filldraw[red] (3.5,0.5) circle (0.08); 
			\filldraw[red] (5,0.5) circle (0.08); 
			\draw[red, dotted, ultra thick] (4.25,0.5) ellipse (0.75cm and 0.5cm); 
			\pic[
			braid/.cd,
			every strand/.style={ultra thick},
			strand 1/.style={black},
			strand 2/.style={black},]
			at (3.75,1.75) {braid={s_1 s_1}};
			\end{tikzpicture}
		}
	\end{center}
	\caption{Construction of series parallel graphs and related links.} \label{fig2}
\end{figure}
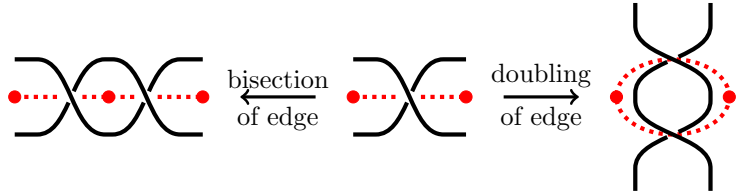

A graph $G$ is \emph{series-parallel} if $G \in \langle T_1 \rangle$, where $T_n$, $n \geq 1$, is a graph consisting of two vertices connected by $n$-multiedge. By the chessboard colorings construction, series-parallel graphs correspond to \emph{alternating} \emph{arborescent} link diagrams, that is, alternating diagrams of links obtained from algebraic tangles~\cite{Sto07}. For algebraic tangles see Conway's paper~\cite{Con}. It was shown by Thistlethwaite~\cite[Th.~4.1]{Thi91} that if a link $K$ admits an alternating diagram which is arborescent, then every alternating diagram of $K$ is arborescent.  A link is said to be arborescent if it admits an arborescent diagram. For definitions and properties of arborescent knots and links see unpublished monograph by Bonahon and Siebenmann~\cite{BS}.  For hyperbolicity of arborescent links see~\cite{FG09} and~\cite{Vo09}. 

Let $t(D)$ be the number of twists in the given alternating twist-reduced arborescent diagram $D$ of a link $K$. If $D$ corresponds to the serial-parallel chessboard coloring graph $G$ differ from $T_2$ and from a join $T_2 * T_2$, then by~\cite[Prop.~6.2]{Sto07}   
\begin{equation}
	\det(K) \geq F_{t(D)+3}, \label{detalg}
\end{equation} 
where $F_k$ denotes the $k$-th Fibonacci numbers defined by  $F_1=F_2=1$ and $F_n=F_{n-1}+F_{n-2}$ for $n>2$.

Using inequalities (\ref{detalg}) and (\ref{eqn:Lackenby}) Stoimenow~\cite[Prop.~6.3]{Sto07} showed that if $K$ is a non-trivial non-split link admitting an alternating arborescent diagram, then 
\begin{equation}
\operatorname{vol}(K) \leq A \cdot \log (\det(K) + B) - C, \label{eqn5}
\end{equation}
where 
\begin{equation}
A = \frac{10v_{\text{\rm tet}}}{\log\frac{1+\sqrt{5}}{2}}, \quad B = \frac{1}{\sqrt{5}} \left(\frac{\sqrt{5}-1}{2}\right)^6, \quad C = 40 v_{\text{\rm tet}} + \frac{10 v_{\text{\rm tet}} \log \frac{1}{\sqrt{5}}}{\log \frac{1+\sqrt{5}}{2}}. \label{eqn55}
\end{equation} 

The following theorem shows that for links whose diagrams have more than eight twists, the bound (\ref{eqn5}) with constants (\ref{eqn55}) can also be improved by using a stronger estimate from the paper~\cite{VeEg24} instead of (\ref{eqn:Lackenby}).

\begin{theorem} \label{thm3}
Let $K$ be a non-trivial non-split link admitting a reduces alternating twist-reduced arborescent diagram with  $t > 8$ twists. Then the following inequality holds 
\begin{equation}
\operatorname{vol}(K) \leq A \cdot \log (\det(K) +D) - E, \label{eqn7}
\end{equation} 
with constants 
\begin{equation} 
A = \frac{10v_{\text{\rm tet}}}{\log\frac{1+\sqrt{5}}{2}}, \quad D = \frac{1}{\sqrt{5}} \left(\frac{\sqrt{5}-1}{2}\right)^{12}, \quad E = 44 v_{\text{\rm tet}} + \frac{10 v_{\text{\rm tet}} \log \frac{1}{\sqrt{5}}}{\log \frac{1+\sqrt{5}}{2}}. \label{eqn8}
\end{equation} 
\end{theorem}

\begin{proof} 
By virtue of~\cite[Prop.~6.2]{Sto07} we have 
$$
\det(K) \geq  F_{t(D)+3}, 
$$
where $F_{t(D) + 3}$ is the $(t(D)+3)$-th Fibonacci number. Recall that the famous Binet's formula gives the $n$-th Fibonacci number in terms of the golden ratio $\phi = \frac{1}{2} (1 + \sqrt{5})$:   
$$
F_n = \frac{\phi^n - (-\phi)^n}{\sqrt{5}} =  \frac{(1 + \sqrt{5})^n - (1-\sqrt{5})^n}{2^n \sqrt{5}}.  
$$
Therefore,  
$$
\det(K) \geq \frac{1}{\sqrt{5}} \left[ \left( \frac{1+\sqrt{5}}{2} \right)^{t(D)+3} - \left( \frac{1-\sqrt{5}}{2}\right)^{t(D)+3} \right].
$$
Since $t(D)>8$, using the formula (\ref{eqn:VE}) we get
$$
\det(K) \geq \frac{1}{\sqrt{5}} \left( \frac{1+\sqrt{5}}{2} \right)^{\frac{\operatorname{vol} K}{10 v_{\text{\rm tet}}} + 4.4} - \frac{1}{\sqrt{5}} \left( \frac{\sqrt{5}-1}{2} \right)^{12}.
$$
Taking the logarithm, we obtain the required inequality. The theorem has been proved. $\Box$ 
\end{proof}
	
\begin{remark} {\rm 
Let us demonstrate numerically the difference between inequalities~(\ref{eqn5}) and~(\ref{eqn7}).  
By replacing the constants with their approximate values, (\ref{eqn5}) gives 
\begin{equation}
	\operatorname{vol} (K) \leq  21.091368  \cdot \log \left( \det (K) + 0.024922 \right) - 23.625039, 
\end{equation}  
while~(\ref{eqn7}) gives   
	\begin{equation}
		\operatorname{vol} (K) \leq 21.091368 \cdot \log (\det (K) + 0.001388) - 27.684806. 
	\end{equation} 	
Finally, we note that in the paper all approximate calculations have been made  to an accuracy of $32$ decimal places and then only $6$ are presented in the final formulae. 
}
\end{remark}
	
\begin{remark} {\rm 
As Stoimenow noticed~\cite[Remark 6.4]{Sto07}, there can be some confusion as to the meaning of alternating arborescent link. There are links which are alternating and arborescent, but whose alternating diagrams are not arborescent. The Borromean rings are the example. If such links are included in the consideration, than the estimate~(\ref{detalg}) should be replaced by the estimate $$\det(K) \geq \frac{8}{17} F_{t(D)+3},$$ 
and coefficients $D$ and $E$ from~(\ref{eqn8}) should be replaced by 
\begin{equation} 
D = \frac{8}{17\sqrt{5}} \left(\frac{\sqrt{5}-1}{2}\right)^{12} \quad \text{and} \qquad E = 44 v_{\text{\rm tet}} + \frac{10 v_{\text{\rm tet}} \log \frac{8}{17\sqrt{5}}}{\log \frac{1+\sqrt{5}}{2}}. 
\end{equation} 
	}
\end{remark}	

\smallskip

\textbf{Acknowledgment.} 
The authors thank Professor Tetsuya Ito for the information on the inequality (\ref{eqn:Ito}) obtained in~\cite[Th.~4.3]{Ito23}.


\end{document}